\documentclass[12pt]{amsart}

\usepackage{amsfonts,amsthm,amsmath,amssymb,latexsym}
\usepackage{graphicx,color}
\usepackage[all]{xy}

\begin{document}

\author{Dragomir \v Sari\' c}

\address{Department of Mathematics,
The Graduate Center of CUNY, 365 Fifth Avenue, Room 4208, New York,
NY 10016-4309 and, \vskip .1 cm Department of Mathematics, 237 Kiely
Hall, 65-30 Kissena Blvd Flushing, NY 11367}

\email{Dragomir.Saric@qc.cuny.edu}

\theoremstyle{definition} 

 \newtheorem{definition}{Definition}[section]
 \newtheorem{remark}[definition]{Remark}
 \newtheorem{example}[definition]{Example}

\newtheorem*{notation}{Notation}  

\theoremstyle{plain}      

 \newtheorem{proposition}[definition]{Proposition}
 \newtheorem{theorem}[definition]{Theorem}
 \newtheorem{corollary}[definition]{Corollary}
 \newtheorem{lemma}[definition]{Lemma}

\def\H{{\mathbf H}}
\def\F{{\mathcal F}}
\def\R{{\mathbf R}}
\def\Q{{\mathbf Q}}
\def\Z{{\mathbf Z}}
\def\E{{\mathcal E}}
\def\N{{\mathbf N}}
\def\X{{\mathcal X}}
\def\Y{{\mathcal Y}}
\def\C{{\mathbf C}}
\def\D{{\mathbf D}}
\def\G{{\mathcal G}}
\def\T{{\mathcal T}}

\title{Circle homeomorphisms and shears}

\subjclass{}

\keywords{}
\date{\today}

\maketitle

\begin{abstract}
We give parameterizations of homeomorphisms, quasisymmetric maps and
symmetric maps of the unit circle in terms of shear coordinates for
the Farey tesselation.
\end{abstract}

\section{Introduction}

The space $Homeo(S^1)$ of orientation preserving homeomorphisms of
the unit circle $S^1$ is a classical topological group which is of
interest in various fields of mathematics \cite{Gh} and in the
bosonic string theory in physics \cite{Wi}, \cite{NS}, \cite{Pe}. An
important subgroup $QS(S^1)$ of quasisymmetric maps of $S^1$ plays a
fundamental role in the Teichm\"uller theory of Riemann surfaces
\cite{A}, \cite{Be}, \cite{GL}. In fact, the {\it universal
Teichm\"uller space} consists of all quasisymmetric maps which fix
three distinguished points on $S^1$ namely it is isomorphic to
$M\ddot{o}b(S^1)\backslash QS(S^1)$, where $M\ddot{o}b(S^1)$ is the
group of (orientation preserving) M\"obius maps which preserve $S^1$
\cite{Be}. The subgroup of symmetric maps $Sym(S^1)$ plays a
prominent role in studying Teichm\"uller spaces of real dynamical
systems \cite{GS}, \cite{EGL}, \cite{J}.

\vskip .2 cm

The main results in this article are explicit parametrizations of
the spaces $M\ddot{o}b(S^1)\backslash Homeo(S^1)$,
$M\ddot{o}b(S^1)\backslash QS(S^1)$ and $M\ddot{o}b(S^1)\backslash
Sym(S^1)$ in terms of shear coordinates for the Farey tesselation of the hyperbolic plane $\H$.
To our best knowledge these are the only known explicit
parametrizations of the above coadjoint orbit spaces. The unit
circle $S^1$ is the boundary at infinity of
$\H$.

\vskip .2 cm

The {\it shear} of the pair $(\Delta ,\Delta_1)$ of ideal
hyperbolic triangles with disjoint interiors and a common boundary
edge $e$ is the signed hyperbolic distance between the orthogonal
projections of the third vertices of $\Delta ,\Delta_1$ onto $e$
(see \cite{Th}, \cite{Bo}, \cite{Pe1}, or Section 3). The {\it
Farey tesselation} $\F$ is a locally finite ideal geodesic
triangulation of $\H$ which is preserved by the hyperbolic
reflections in edges of $\F$ (see, for example, \cite{Pe}). The
set of edges of $\F$ is naturally partitioned into Farey
generations (see \cite{Pe} or Section 3). The shear of each pair
of adjacent complementary triangles of $\F$ is zero.

\vskip .2 cm

A homeomorphism $h:S^1\to S^1$ induces a map $s_h:\F\to\R$, called
the {\it shear map}, from the Farey tesselation $\F$ to the set $\R$
of real numbers as follows. Each $e\in\F$ is the common boundary
side of a pair $(\Delta ,\Delta_1)$ of complementary triangles of
$\F$. We define $s_h(e)$ to be the shear of the image pair
$(h(\Delta ),h(\Delta_1))$. It was known to be a challenging problem
to characterize which maps $s:\F\to\R$ arise from homeomorphisms and
which arise from quasisymmetric maps of $S^1$. We answer these
questions below. (For a punctured surfaces $S'$, the Teichm\"uller
space $T(S')$ is parameterized using shears by Thurston \cite{Th}
and Penner \cite{Pe1}; in the case of a closed surfaces $S$,
Thurston \cite{Th1} and Bonahon \cite{Bo} gave a parameterization of
$T(S)$ using shears on locally infinite tesselations.)

\vskip .2 cm

A {\it fan of geodesics} in $\F$ with {\it tip} $p\in S^1$ consists
of all edges of $\F$ which have one endpoint $p$. Each fan in $\F$
has a natural ordering as follows. Fix a horocycle $C$ with center
at $p$ whose orientation is such that the corresponding horoball is
to the left of $C$. If $e,e'$ are two geodesics with a common
endpoint $p$, then we define $e<e'$ if point $e\cap C$ comes before
point $e'\cap C$ on $C$, otherwise $e'<e$. The natural ordering on a
fan induces a bijective correspondence of the geodesics of the fan
with the integers $\Z$, and any two such correspondences differ by a
translation in $\Z$. For each fan of $\F$ we fix one such
correspondence.

\vskip .2 cm

\noindent {\bf Theorem A.} {\it A shear map $s:\F\to \R$ is induced
by a quasisymmetric map of $S^1$ if and only if there exists $M\geq
1$ such that for each fan of geodesics $e_n\in\F$, $n\in\Z$, and for
all $m,k\in\Z$, we have
$$
\frac{1}{M}\leq \frac{e^{\frac{s_m}{2}}+e^{\frac{s_m}{2}+s_{m+1}}+\cdots
+e^{\frac{s_m}{2}+s_{m+1}+\cdots
+s_{m+k}}}{e^{-\frac{s_m}{2}}+e^{-\frac{s_m}{2}-s_{m-1}}+\cdots
+e^{-\frac{s_m}{2}-s_{m-1}-\cdots -s_{m-k}}}\leq M,
$$
where $s_n=s(e_n)$.

\noindent Moreover, $s:\F\to\R$ is induced by a symmetric map of $S^1$ if and only if
$$
\frac{e^{\frac{s_m}{2}}+e^{\frac{s_m}{2}+s_{m+1}}+\cdots
+e^{\frac{s_m}{2}+s_{m+1}+\cdots
+s_{m+k}}}{e^{-\frac{s_m}{2}}+e^{-\frac{s_m}{2}-s_{m-1}}+\cdots
+e^{-\frac{s_m}{2}-s_{m-1}-\cdots -s_{m-k}}}\rightrightarrows 1
$$
as the Farey generations of $e_{m-k}$ and $e_{m+k}$ go to infinity. }

\vskip .2 cm

For a fan of $\F$ with tip $p$, we define
$$
s(p;m,k)=\frac{e^{\frac{s_m}{2}}+e^{\frac{s_m}{2}+s_{m+1}}+\cdots
+e^{\frac{s_m}{2}+s_{m+1}+\cdots
+s_{m+k}}}{e^{-\frac{s_m}{2}}+e^{-\frac{s_m}{2}-s_{m-1}}+\cdots
+e^{-\frac{s_m}{2}-s_{m-1}-\cdots -s_{m-k}}}
$$
for $m,k\in\Z$. Let $C$ be a horoball with center at $h(p)$ where
$h$ is a quasisymmetric map which induces $s$. Then $s(p;m,k)$
represents the ratio of the length of the horocyclic arc on $C$
between $h(e_{m+k})$ and $h(e_m)$ to the length of the horocyclic
arc on $C$ between $h(e_{m-k})$ and $h(e_m)$. Define
$$M_s(p)=\sup_{m,k\in\Z} s(p;m,k).$$ If $M_s(p)<\infty$ then we
say that $s$ satisfies $M_s(p)$-{\it condition at the fan with
tip} $p$. The above theorem states that a shear map $s:\F\to\R$
induces a quasisymmetric map if and only if
\begin{equation}
\label{M-cond}
M_s=\sup_p M_s(p)<\infty
\end{equation}
where the supremum is over all fans of $\F$.

\vskip .2 cm

It is quite surprising that the characterization of shears which
give rise to quasisymmetric maps is so simple. The
$M_s(p)$-condition is localized in a single fan of geodesics with
tip $p$ and the only additional information is that single
$M_s=\sup_pM_s(p)$ works for all fans simultaneously. In
particular, there is no information as how shears on close by
geodesics not belonging to a single fan relate to each other.

\vskip .2 cm

We now interpret the Teichm\"uller topology of the universal
Teichm\"uller space $M\ddot{o}b(S^1) \backslash QS(S^1)$ within the
framework of Theorem A. That theorem parametrizes $M\ddot{o}b(S^1)
\backslash QS(S^1)$ by the space $\X$ of all shear maps $s:\F\to\R$
which satisfy (\ref{M-cond}). We use $s(p;m,k)$ to introduce a
natural topology on $\X$ such that the parametrization of
$M\ddot{o}b(S^1)\backslash QS(S^1)$ by $\X$ is a homeomorphism. For
$s_1,s_2\in\X$ define
$$
M_{s_1,s_2}(p)=\sup_{m,k}\Big{(}\max\Big{\{}
\frac{s_1(p;m,k)}{s_2(p;m,k)},\frac{s_2(p;m,k)}{s_1(p;m,k)}\Big{\}}\Big{)}.
$$

\vskip .2 cm

\noindent {\bf Theorem B.} {\it Let $h_n,h\in
M\ddot{o}b(S^1)\backslash QS(S^1)$. Then $h_n\to h$ as $n\to\infty$
in the Teichm\"uller topology if and only if
$M_{s,s_n}=\sup_pM_{s,s_n}(p)\to 1$ as $n\to\infty$.}

\vskip .2 cm

Surprisingly enough the characterization of homeomorphisms involves
more information then the parametrization of quasisymmetric
homeomorphisms given by Theorem A. A {\it chain of geodesics} in
$\F$ is a sequence $e_n\in\F$ of distinct edges such that $e_n$ and
$e_{n+1}$ share a common endpoint for all $n\in\N$.

\vskip .2 cm

\noindent {\bf Theorem C.} {\it A shear map $s:\F\to \R$ is induced
by a homeomorphism of $S^1$ if and only if for each chain
$e_n\in\F$, $n\in\N$, we have
$$
\sum_{n=1}^{\infty} e^{s_1^n+s_2^n+\cdots +s_n^n}=\infty
$$
where $s_i^n=\pm s(e_i)$. More precisely if $e_n<e_{n+1}$ then
$s_n^n=s(e_n)$; otherwise $s_n^n=-s(e_n)$. For $n>1$ and $i<n$,
$s_i^n=s(e_i)$ if either $e_i<e_{i+1}$ and the number of times we
change fans from $e_{i}$ to $e_{n+1}$ is even, or $e_i>e_{i+1}$ and
the number of times we change fans is odd; otherwise
$s_i^n=-s(e_i)$.}

\vskip .2 cm

A locally finite ideal triangulation of $\H$ with a distinguished
oriented edge is called a {\it tesselation}. The space of all
tesselations is isomorphic to the space $Homeo(S^1)$ by assigning
to a tesselation $\tau$ a homeomorphism of $S^1$ (called the {\it
characteristic map}) which maps the Farey tesselation $\F$ to the
tesselation $\tau$ of $\H$ such that a distinguished oriented edge
of $\F$ is mapped onto the distinguished oriented edge of $\tau$
(see Penner \cite{Pe}). A {\it decorated tesselation} is a
tesselation together with an arbitrary assignment of a horocycle
at each vertex of the tesselation (see \cite{Pe}).

\vskip .2 cm

Let $C_1$ and $C_2$ be two horocycles with different centers and
let $g$ be the geodesic whose endpoints are at the centers of
$C_1$ and $C_2$. Then the {\it lambda length} $\lambda (g)$ of $g$
is defined by
$$
\lambda (g)=e^{-2\delta (C_1,C_2)}
$$
where $\delta (C_1,C_2)$ is the signed hyperbolic distance between
$G_1=g\cap C_1$ and $G_2=g\cap C_2$. The sign of $\delta
(C_1,C_2)$ is positive if the geodesic arc between $G_1$ and $G_2$
is outside $C_1$, otherwise the sign is negative. Let $g,g_1$ be a
wedge of geodesics in $\H$ and let $C$ be a horocycle with center
at the common endpoint of $g$ and $g_1$. The {\it horocyclic
length} $\alpha (g,g_1)$ of the wedge $g,g_1$ is the length of the
arc of $C$ between $g$ and $g_1$. A decorated tesselation
$\tilde{\tau}$ determines an assignment of lambda lengths to the
edges of $\tau$ and of horocyclic lengths to the wedges of $\tau$.
This in turn defines an assignment of lambda lengths to the edges
of the Farey tesselation $\F$ by the pull-back with the
characteristic map as well as the assignment of horocyclic lengths
to the wedges of $\F$ (see Penner \cite{Pe}, \cite{Pe1}).

\vskip .2 cm

Two decorated tesselations $\tilde{\tau}_1$ and $\tilde{\tau}_2$
induce the same lambda lengths on $\F$ if and only if
$\tilde{\tau}_1$ is the image under an element of $M\ddot{o}b(S^1)$
of $\tilde{\tau}_2$. It is clear that not every assignment of lambda
lengths on the Farey tesselation $\F$ will give a decorated
tesselation such that the characteristic map is a homeomorphism of
$S^1$. In fact the underlying tesselation is not in general an
ideal triangulation of $\H$. Penner \cite{Pe} posed the problem of
determining which lambda lengths will give characteristic maps that
are homeomorphisms or quasisymmetric maps of $S^1$. Penner and
Sullivan \cite[Theorem 6.4]{Pe} showed that if lambda lengths are
``pinched'' namely if there is $K\geq 1$ such that $1/K\leq\lambda
(e)\leq K$ for all $e\in\F$ then the characteristic map is
quasisymmetric. We find necessary and sufficient conditions on the
lambda lengths such that the characteristic maps are homeomorphisms,
quasisymmetric or symmetric maps of $S^1$.

\vskip .3 cm

\noindent {\bf Theorem D.} {\it A lambda length function $\lambda
:\F\to\R^{+}$ induces a homeomorphism of $S^1$ if and only if for
each chain of edges $e_n\in\F$, $n\in\N$, we have
$$
\sum_{n=1}^{\infty}\Big{(}\lambda_n^{-\frac{1}{2}}
\lambda_{n-1}^{\frac{1}{2}}\cdots \lambda_1^{\frac{(-1)^n}{2}}\Big{)}\alpha_n=\infty
$$
where $\lambda_i=\lambda(e_i)$ and $\alpha_n$ is the horocyclic
length of the wedge bounded by $e_n$ and $e_{n+1}$.}

\vskip .3 cm

In the above theorem we used horocyclic length $\alpha_n$ together
with the lambda lengths. We note that horocyclic lengths are
expressed as rational functions of lambda lengths (see Penner
\cite{Pe1}, \cite[Section 6]{Pe}). Indeed, if $g_1,g_2,g_3$ are
edges of an ideal triangle with decorations then by \cite{Pe1} we
have
$$
\alpha (g_1,g_2)=\frac{2\lambda (g_3)}{\lambda (g_1)\lambda (g_2)}.
$$
Thus the series in the above theorem is completely determined in terms of lambda lengths.

\vskip .2 cm

The following theorem gives necessary and sufficient conditions on
horocyclic lengths such that the characteristic maps are
quasisymmetric and symmetric. We note that it is possible to express
the same condition in terms of lambda lengths using the formula above.

\vskip .3 cm

\noindent {\bf Theorem E.} {\it  A lambda length function $\lambda
:\F\to\R$ induces a quasi-symmetric map of $S^1$ if and only if
there exists $K\geq 1$ such that for each fan $e_n\in\F$,
$n\in\Z$, and for all $m\in\Z$ and $k\in\N$ we have
$$
\frac{1}{K}\leq\frac{\alpha (e_m,e_{m+1})+\alpha (e_{m+1},e_{m+2})+\cdots +\alpha (e_{m+k},e_{m+k+1})}
{\alpha (e_m,e_{m-1})+\alpha (e_{m-1},e_{m-2})+\cdots +\alpha (e_{m-k},e_{m-k-1})}\leq K.
$$

\noindent
Moreover, $\lambda :\F\to\R$ induces a symmetric
map of $S^1$ if and only if
$$
\frac{\alpha (e_m,e_{m+1})+\alpha (e_{m+1},e_{m+2})+\cdots +\alpha (e_{m+k},e_{m+k+1})}
{\alpha (e_m,e_{m-1})+\alpha (e_{m-1},e_{m-2})+\cdots +\alpha (e_{m-k},e_{m-k-1})}\to 1
$$
as the Farey generations of $e_{m+k}$ and $e_{m-k}$ go to infinity independently of the fan.
}

\section{Quasisymmetric maps and barycentric extension}

In the rest of the paper the hyperbolic plane is identified with
the upper half-plane model $\H :=\{ z=x+iy|\ y>0\}$ endowed with
the metric $\rho (z)=\frac{|dz|}{y}$. The boundary at infinity
$\partial_{\infty}\H=\hat{\R}=\R\cup\{\infty\}$ is naturally
identified with the unit circle $S^1$. Any two identifications of
$\hat{\R}$ and $S^1$ differ by the postcomposition by a M\"obius
map of $S^1$. We choose $0$, $1$ and $\infty$ to be the three
distinguished points on $\hat{\R}$.

\vskip .2 cm

Let $h:\hat{\R}\to\hat{\R}$ be a homeomorphism that fixes $\infty$
and let $M\geq 1$. Then $h:\hat{\R}\to\hat{\R}$ is said to be
$M$-{\it quasisymmetric} if
$$
\frac{1}{M}\leq\frac{h(x+t)-h(x)}{h(x)-h(x-t)}\leq M
$$
for all $x\in\R$ and $t>0$ (see \cite{A}).

\vskip .2 cm

The {\it universal Teichm\"uller space} $T(\H )$ is the set of all
quasisymmetric maps of $\hat{\R}$ that fix $0$, $1$ and $\infty$.
A sequence $h_n\in T(\H )$ converges to the basepoint $id\in T(\H
)$ in the {\it Teichm\"uller topology} if $h_n$ are
$M_n$-quasisymmetric with $M_n\to 1$ as $n\to\infty$. A sequence
$h_n\in T(\H )$ converges to $h\in T(\H )$ in the {\it
Teichm\"uller topology} if $h_n\circ h^{-1}\to id$ as $n\to\infty$
in the above sense.

\vskip .2 cm

A quasisymmetric map $h:\hat{\R}\to\hat{\R}$ extends to a
quasiconformal map $f:\H\to\H$, and conversely a quasiconformal
map $f:\H\to\H$ extends by continuity to a quasisymmetric map
$h:\hat{\R}\to\hat{\R}$ (see \cite{A}). The extension of
$h:\hat{\R}\to\hat{\R}$ to a quasiconformal map of $\H$ is not
unique. Douady and Earle defined a particularly nice extension
operator from quasisymmetric maps of $\hat{\R}$ into
quasiconformal maps of $\H$ called the {\it barycentric extension}
(see \cite{DE}).

\vskip .2 cm

For a homeomorphism $h:\hat{\R}\to\hat{\R}$, denote by
$ex(h):\H\to\H$ its barycentric extension introduced in \cite{DE}.
We recall several properties of $ex(h)$ that are obtained by
Douady and Earle \cite{DE}. The barycentric extension $ex(h)$ is a
real-analytic diffeomorphism of $\H$ which is quasiconformal if
and only if $h$ is quasisymmetric. Moreover, the extension is
conformaly natural in the sense that $ex(A\circ h\circ B)=A\circ
ex(h)\circ B$ for all $A,B\in PSL_2(\R )$ and for all
homeomorphisms $h:\hat{\R}\to\hat{\R}$. In addition, if $h_n\to h$
as $n\to\infty$ pointwise on $\hat{\R}$ then $ex(h_n)\to ex(h)$ as
$n\to\infty$ in the $C^{\infty}$-topology on $C^{\infty}$ maps of
$\H$. In particular, Beltrami coefficients $\mu (ex(h_n))$ of
$ex(h_n)$ converge uniformly on compact subsets of $\H$ to the
Beltrami coefficient $\mu (ex(h))$ of $ex(h)$.

\begin{remark}
For our purposes the barycentric extension serves quite well. Kahn
and Markovic \cite{KM} constructed another quasiconformal extension
in the case when the quasisymmetric maps are invariant under
co-finite Fuchsian group in order to be able to estimate the norm of
the corresponding Beltrami coefficient.
\end{remark}

\vskip .2 cm

The following lemma is obtained by Markovic \cite{Mar} (see also Douady-Earle \cite{DE} and
Abikoff-Earle-Mitra \cite{AEM}).

\begin{lemma}
\label{baryc} Let $h_n:\hat{\R}\to\hat{\R}$ be a sequence of
homeomorphisms which fix $0$, $1$ and $\infty$, and let $\mu_n$ be
Beltrami coefficients of the barycentric extensions $ex(h_n)$ of
$h_n$. If there exists $c_0\geq 1$ such that
$$-c_0\leq h_n(-1)\leq -\frac{1}{c_0}$$
then there exists a neighborhood
$U$ of the imaginary unit $i\in\H$ and a constant $0<c<1$ such that
$$
\|\mu_n|_U\|_{\infty}\leq c<1
$$
for all $n$.
\end{lemma}

\begin{proof}
We note that the angle distance with respect to $i\in\H$ between all pairs of consecutive points in
$\{ \infty ,-1,0,1\}\subset \hat{\R}$ is bounded below by a constant less than $\pi$ and bounded
above by $\pi$. Then \cite[Lemma 3.6]{Mar} directly implies the desired
conclusion.
\end{proof}

\section{The Farey tesselation and the shear map}

Let $\Delta_0$ be an ideal geodesic triangle in $\H$ with vertices
$0$, $1$ and $\infty$. Let $\Gamma$ be the group generated by
hyperbolic reflections in the sides of $\Delta_0$. The Farey
tesselation $\F$ is an ideal triangulation of $\H$ which is the
$\Gamma$-orbit of the boundary sides of $\Delta_0$. In other words,
each edge in $\F$ is obtained by applying finitely many inversions
in the sides $\Delta_0$ to an edge of $\Delta_0$ (see, for example,
\cite{Pe}). The set of endpoints of the edges in $\F$ is
$\hat{\mathbf{Q}}=\mathbf{Q}\cup\{\infty\}$.

\vskip .2 cm

We define {\it Farey generation} of edges of $\F$ as follows.
A boundary edge of $\Delta_0$ has Farey generation $0$. If
a boundary edge of $\F$ is obtained by $n$ reflections of an edge of
generation $0$
(where $n$ is the smallest such number) then its Farey generation
is $n$.

\vskip .2 cm

Let $(\Delta_1,\Delta_2)$ be a pair of ideal triangles in $\H$ with
disjoint interiors and a common boundary side. Let $A\in PSL_2(\R )$
be the unique M\"obius map that sends  $\Delta_1$ onto the triangle with
vertices $-1$, $0$ and $\infty$, and that sends the common boundary
side of $(\Delta_1,\Delta_2)$ onto the geodesic with vertices $0$
and $\infty$. Then $A(\Delta_2)$ has vertices $0$, $e^r$ and
$\infty$ for some $r\in\R$. The {\it shear} of the pair of triangles
$(\Delta_1,\Delta_2)$ is by definition equal to $r$.
Alternatively, the shear of a pair $(\Delta_1,\Delta_2)$ of adjacent
triangles is the signed distance of the projections onto common
boundary side $e$ of vertices of $\Delta_1$ and $\Delta_2$ opposite
$e$, where $e$ is oriented to the left as seen from $\Delta_1$. Note
that the shear of $(\Delta_1,\Delta_2)$ is equal to the shear of
$(\Delta_2,\Delta_1)$. For example, any two adjacent triangles in the
complement of the Farey tesselation $\F$ have shear $0$.

\vskip .2 cm

Let $h:\hat{\R}\to\hat{\R}$ be a homeomorphism. Every geodesic of
$\H$ has exactly two distinct ideal endpoints on $\hat{\R}$ and,
conversely every two points on $\hat{\R}$ determine a geodesic in
$\H$. Thus, the space $\G$ of (oriented) geodesics in $\H$ is
identified with the set of pairs of distinct points in $\hat{\R}$.
Therefore, the homeomorphism $h:\hat{\R}\to\hat{\R}$ extends to a
homeomorphism $h:\G\to\G$ of the space of geodesics $\G$. In
particular, $h(\F )$ is an ideal triangulation of $\H$ whose
complementary triangles are $h(\Gamma (\Delta_0))$.

\begin{definition}
Let $h:\hat{\R}\to\hat{\R}$ be a homeomorphism. An edge $e\in\F$ is
on the boundary of exactly two complementary triangles $\Delta_1,\Delta_2$.
Then we assign to $e\in\F$ the shear of the pair $(
h(\Delta_1),h(\Delta_2))$ of triangles in $h(\Gamma (\Delta_0))$. This determines
a map
$$s_h:\F\to\R$$ which is called the {\it shear map} of $h$.
\end{definition}

If we are given a shear between two adjacent triangles and the
position of one of the triangles, the other triangle is uniquely
determined. More generally, a pair of adjacent triangles with an
assigned shear is determined up to a M\"obius map because any
ideal hyperbolic triangle can be mapped onto any other ideal
hyperbolic triangle by a M\"obius map.

\vskip .2 cm

If $h:\hat{\R}\to\hat{\R}$ fixes $0$, $1$ and $\infty$, then it is
uniquely determined by the shear map $s_h:\F\to\R$. Given a shear
map $s:\F\to\R$ there exists a unique injective map $h_s$ from the
vertices $\hat{\Q}\subset\hat{\R}$ of the Farey tesselation $\F$
into $\hat{\R}$ such that $h_s$ fixes $0$, $1$ and $\infty$. The map
$h_s$ realizes the shear map $s$ and it is called a {\it characteristic map}
of $s$ (see \cite{Pe} or next section for its definition).

\section{Homeomorphisms and shears}

We characterize shear maps $s:\F\to\R$ whose characteristic maps
continuously extend to homeomorphisms
of $\hat{\R}$. An arbitrary map $s:\F\to\R$ induces a cocycle map $H_s:\H\to\H$
which is piecewise M\"obius as follows. Let $H_s|_{\Delta_0}=id$.
For any other complementary triangle $\Delta\in\Gamma (\Delta_0)$, let $l$ be the geodesic arc
connecting the center of $\Delta_0$ to the center of $\Delta$. Let
$\{ e_1,e_2,\ldots ,e_n\}$ be the edges in $\F$ which intersect $l$
in the given order such that $e_1$ is a boundary side of $\Delta_0$
and $e_n$ is a boundary side of $\Delta$. We orient $e_i$ to the
left as seen from $\Delta_0$. Then we set
$H_s|_{\Delta}=T^{s(e_1)}_{e_1}\circ
T^{s(e_{2})}_{e_{2}}\circ\cdots \circ T^{s(e_n)}_{e_n}$, where
$T^{s(e_i)}_{e_i}$ is the hyperbolic translation with the oriented
axis $e_i$ and the signed translation length $s(e_i)$. The map $H_s$
is not well-defined on the edges $\F$ since each edge $e$ is on the
boundary of exactly two complementary triangles $\Delta_e^1$ and $\Delta_e^2$.
We choose $H_s|_e$ to be either $H_s|_{\Delta_e^1}$ or
$H_s|_{\Delta_e^2}$. The cocycle map $H_s$ preserves separation
properties of the triples of complementary triangles of $\F$. Therefore, $H_s$ extends to a
monotone map $h_s:\hat{\mathbf{Q}}\to\hat{\R}$ which is called {\it characteristic
map} of $s:\F\to\R$ (see Penner \cite{Pe}).

\begin{proposition} With the above notation, the characteristic map
$h_s:\hat{\mathbf{Q}}\to\hat{\R}$ extends by continuity to a
homeomorphism of $\hat{\R}$ if and only if $H_s:\H\to\H$ is
surjective.
\end{proposition}

\begin{proof}
Since $h_s:\hat{\mathbf{Q}}\to\hat{\R}$ is order preserving on the
dense subset $\hat{\mathbf{Q}}$ of $\hat{\R}\equiv S^1$, it follows
that if $h_s$ can be extended to a continuous map on $\hat{\R}$ then
the extension is a homeomorphism.

\vskip .2 cm

If $H_s:\H\to\H$ is not onto, then there exists a maximal half-plane
$P$ not contained in $H_s(\H )$. It follows that the image
$h_s(\hat{\mathbf{Q}})$ does not intersect the interior of the
interval on $\hat{\R}$ which is the boundary at infinity of $P$.
Therefore, the map $h_s:\hat{\mathbf{Q}}\to\hat{\R}$ cannot be
extended to a homeomorphism of $\hat{\R}$.

\vskip .2 cm

Assume that $H_s:\H\to\H$ is onto. Let
$x\in\hat{\R}\setminus\hat{\mathbf{Q}}$. We need to show that $h_s$
extends to $x$. Let $P_i$ be a decreasing sequence of half-planes
with boundary sides $e_i\in\F$ that accumulate at $x$, namely
$\bigcap_i\overline{P_i}=x$. Since $H_s$ is order preserving on
triples of complementary triangles of $\F$, it follows that $H_s(P_i)$ is a
decreasing sequence of half-planes. If
$\bigcap_iH_s(P_i)\neq\emptyset$ then $H_s(\H )\neq \H$, namely
$H_s(\H )\cap (\bigcap_iH_s(P_i))=\emptyset$. Thus
$\bigcap_iH_s(P_i)=\emptyset$ and $\bigcap_i\overline{H_s(P_i)}$ is
a single point $y\in\hat{\R}$. Then $h_s$ extends to $x$ by
continuity such that $h_s(x)=y$.
\end{proof}

\vskip .2 cm

\noindent {\it Proof of Theorem C.} Using the above proposition we
determine which shear maps induce homeomorphisms of $\hat{\R}$.
Assume that $H_s:\H\to\H$ is not onto. Then there exists a maximal
half-plane $P$ of $\H$ which is not in the image of $H_s$. Let $l$
be the boundary geodesic of the half-plane $P$. Then there exists a
chain $e_n\in\F$ such that $H_s(e_n)\to l$ as $n\to\infty$. There
are two possibilities for the sequence $e_n$. Either all $e_n$'s
share a common endpoint $x\in\hat{\Q}\subset\hat{\R}$ for $n\geq
n_0$ namely the sub-chain $e_n$, for $n\geq n_0$, is a part of a
single fan, or $e_n$'s accumulate to a point
$x\in\hat{\R}\setminus\hat{\Q}$ (which is equivalent to saying that
no infinite subsequence of $e_n$'s shares a common endpoint i.e. no
tail of $e_n$'s is a part of a single fan). In both cases the
existence of the half-plane $P$ is equivalent to the statement that
$h_s$ does not extend to a continuous map at $x\in\hat{\R}$.

\vskip .2 cm

Assume that we are in the first case. By pre-composition with an element
of $PSL_2(\Z )$, we can assume that $x=\infty$. In addition, we can
assume that $H_s$ fixes $0$, $1$ and $\infty$ by post-composing with
an element of $PSL_2(\R )$. Then $l$ has one endpoint $x=\infty$ and
the other endpoint $\bar{y}\in\R$ with either $\bar{y}>1$ or
$\bar{y}<0$. If $\bar{y}>1$, then
\begin{equation} \label{finite1} \bar{y}=1+\sum_{n=1}^{\infty}e^{s(e_1)+\cdots +s(e_n)},
\end{equation} where
$e_i\in \F$ is a geodesic with endpoints $i$ and
$\infty$ for $i\in\mathbf{N}$. If $\bar{y}<0$ then
\begin{equation} \label{finite2}
\bar{y}=-\sum_{n=0}^{-\infty}e^{-s(e_0)-\cdots -s(e_n)},
\end{equation} where
$e_i\in\mathcal{F}$ is a geodesic with endpoints
$i$ and $\infty$ for $i\in\mathbf{Z}^{-}\cup\{ 0\}$. Since $e_i$'s belong
to a single fan, the number of times we change fans from $e_{i}$ to $e_{n+1}$
is zero. Thus $s_i^n=s(e_i)$ for $i>0$ and $s_i^n=-s(e_i)$ for $i\leq 0$.
Therefore, $h_s$ is continuous at $x\in\hat{\R}$ if and only if
the series in (\ref{finite1}) and the series in (\ref{finite2}) diverge.

\vskip .2 cm

Assume now that we are in the second case. Namely, the chain
$e_n$ does not have a subsequence which shares a common endpoint
and $e_n$'s accumulate at $x\in\hat{\R}\setminus\hat{\Q}$. In other words,
no tail of $e_n$'s is in a single fan. The part of $\H$ bounded by $e_n$ and $e_{n+1}$ is
called a {\it hyperbolic wedge}.

\vskip .2 cm

Given a hyperbolic wedge, there is a unique foliation of the wedge
by horocyclic arcs which lie on horocycles with centers at the
common endpoint of the two boundary geodesics of the wedge. Consider
the wedges whose boundaries are the adjacent geodesics in the chain
$h_s(e_n)$ and foliate each wedge by horocyclic arcs as above. Fix a
point $P_1\in h_s(e_1)$ and denote by $l(P_1)$ the leaf of the
horocyclic foliation of the union of wedges that starts at $P_1$.
Let $W_n$ be the hyperbolic wedge bounded by $h_s(e_n)$ and
$h_s(e_{n+1})$. We choose $P_1$ such that the length of $l(P_1)\cap
W_1$ is $e^{s_1^1}$, where $s^1_1=s(e_1)$ if $e_1<e_2$, otherwise
$s^1_1=-s(e_1)$ (see Figure \ref{fig1}).

\begin{figure}
\centering
\includegraphics[scale=0.5]{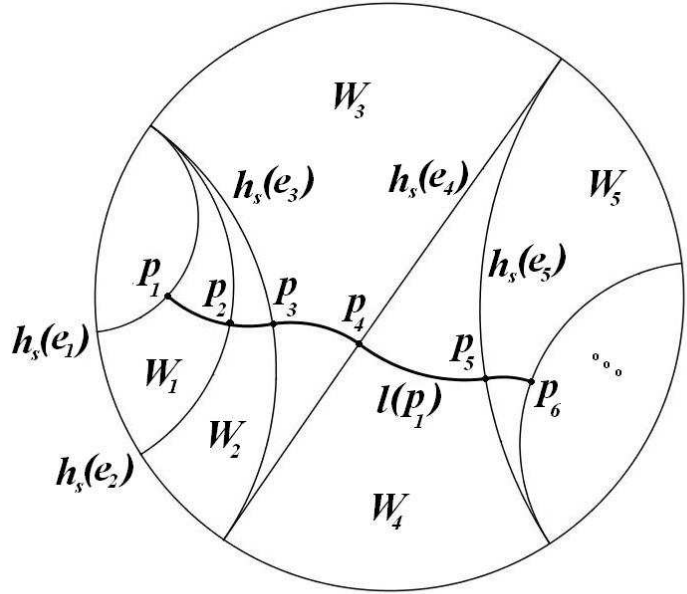}
\caption{The leaf $l(P_1)$ of the foliation of $\cup_{n}W_n$ by
horocycles.} \label{fig1}
\end{figure}

\begin{proposition} Under the above notation,
the map $h_s$ continuously extends to
$x\in\hat{\R}\setminus\hat{\mathbf{Q}}$ if and only if the leaf
$l(P_1)$ is of infinite length.
\end{proposition}

\begin{proof}
Note that $h_s$ extends by continuity to $x\in\hat{\R}$ if and
only if $h_s(e_n)$ do not accumulate in $\H$.

\vskip .2 cm

Assume that $h_s$ extends continuously to $x\in\hat{\R}$. Then
$h_s(e_n)$ do not accumulate in $\H$. Therefore, the arc $l(P_1)$
accumulates at $\partial\H$ and it is necessarily of infinite
length.

\vskip .2 cm

It remains to show that if $l(P_1)$ is of infinite length then $h_s$
extends to $x\in\hat{\R}$ by continuity. Assume on the contrary
that $h_s$ does not extend to $x\in\hat{\R}$. This implies that
$h_s(e_n)$ accumulate at a geodesic $g\subset\H$. We need to show
that $l(P_1)$ has finite length.

\vskip .2 cm

Let $a$ be the geodesic arc which connects $h_s(e_1)$ with $g$ and
that is orthogonal to both $h_s(e_1)$ and $g$. All the geodesics
of the sequence $h_s(e_n)$ for $n\geq 2$ lie between $h_s(e_1)$
and $g$, and they intersect $a$. The angle of the intersection
between $h_s(e_n)$ and $a$ is necessarily bounded away from $0$.
We show that the length of $l(P_1)$ is comparable to the length of
$a$ which finishes the proof.

\vskip .2 cm

Consider a hyperbolic wedge $W_n$ bounded with $h_s(e_n)$ and
$h_s(e_{n+1})$. Let $a_n=a\cap W_n$, and let $P'_n=a\cap h_s(e_n)$.
Then $P'_n$ and $P'_{n+1}$ are the endpoints of $a_n$. Let
$P_n=l(P_1)\cap h_s(e_{n})$ and let $P_n''$ be the endpoint of the
horocyclic arc in the wedge $W_{n-1}$ whose initial point is
$P_{n-1}'$ (see Figure \ref{fig2}). Let $d_n$ be the geodesic arc
with endpoints $P_n$ and $P_n'$, and let $d_n'$ be the geodesic arc
with endpoints $P_n'$ and $P_n''$. Consider the hyperbolic triangle
with vertices $P_{n-1}'$, $P_n'$ and $P_n''$. Since the angle at
$P_n''$ is bounded away from $0$ (by the uniform bound on the length
of each $a_n$), it follows from the hyperbolic sine formula that
there exists $C>0$ such that $|d_n'|\leq C\cdot |a_{n-1}|$, where
$|d_n'|, |a_n|$ are the lengths of $d_n',a_n$, respectively. In
addition, $|d_n|\leq |d_{n-1}|+|d_n'|$ follows by the definition of
$l(P_1)$.

\begin{figure}
\centering
\includegraphics[scale=0.35]{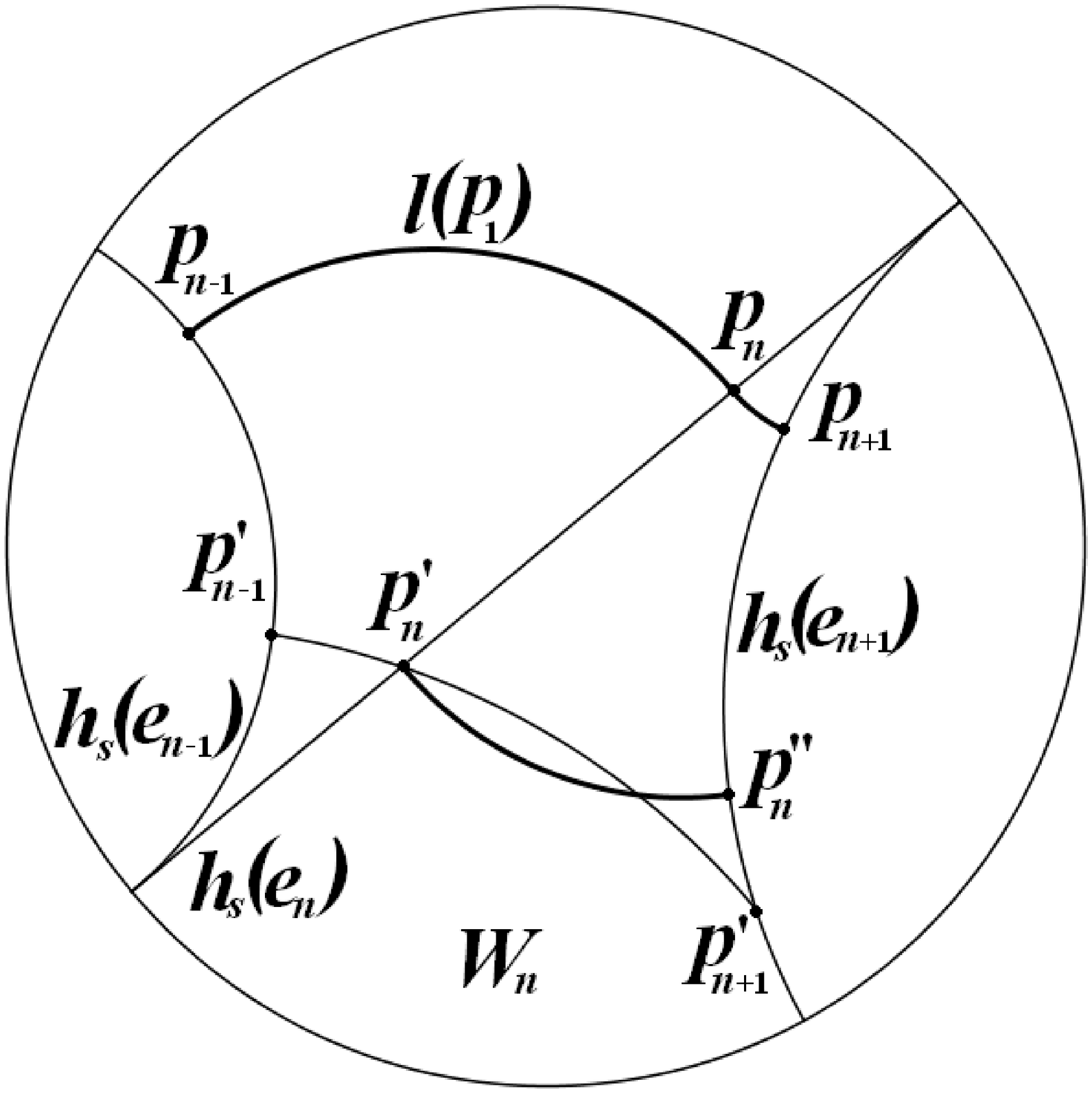}
\caption{The points $P_n$, $P_n'$ and $P_n''$.} \label{fig2}
\end{figure}

\vskip .2 cm

The above two estimates show that $\sum_{n\in\mathbf{N}}|d_n|\leq
|d_1|+C\sum_{n\in\mathbf{N}}|a_n|=C_1|a|<\infty$. This implies that $l(P_1)$
stays a bounded distance from $a$. Thus the length of $l(P_1)\cap
W_n$ and the length $a_n$ are comparable to a multiplicative
constant. Therefore $l(P_1)$ has finite length.
\end{proof}

\vskip .2 cm

We use the above proposition to find a condition on the shear map
$s$ such that $h_s$ has continuous extension to $x$. We compute the
length of the above leaf $l(P_1)$ in terms of the shear map
$s:\F\to\R$. Let $l_n$ be the length of the horocyclic arc $l(P_1)\cap W_n$ in the
wedge $W_n$ between $H_s(e_n)$ and $H_s(e_{n+1})$. If
$e_n$, $e_{n+1}$ and $e_{n+2}$ share a common endpoint, then an
elementary hyperbolic geometry and the definition of $H_s$ show that
the length of $l(P_1)\cap W_{n+1}$ in the wedge $W_{n+1}$ bounded by $H_s(e_{n+1})$
and $H_s(e_{n+2})$ is $l_ne^{s(e_{n+1})}$ if $e_{n+1}<e_{n+2}$,
and the length is $l_ne^{-s(e_{n+1})}$ if $e_{n+2}<e_{n+1}$. If $e_n$, $e_{n+1}$ and
$e_{n+2}$ do not share a common endpoint, then the length of $l(P_1)\cap W_{n+1}$
in the wedge $W_{n+1}$ between $H_s(e_{n+1})$ and $H_s(e_{n+2})$ is
$l_n^{-1}e^{s(e_{n+1})}$ if $e_{n+1}<e_{n+2}$, and the length is
$l_n^{-1}e^{-s(e_{n+1})}$ if $e_{n+1}>e_{n+2}$. We choose $P_1\in H_s(e_1)$
such that $l_1=e^{s_1^1}$.

\vskip .2 cm

We show that $l_n=e^{s_1^n+s_2^n+\cdots +s_n^n}$ by induction which
finishes the proof. Note that the choice of $P_1\in H_s(e_1)$ is
such that $l_1=e^{s_1^1}$. Assume that $l_n=e^{s_1^n+s_2^n+\cdots
+s_n^n}$ and we need to show that
$l_{n+1}=e^{s_1^{n+1}+s_2^{n+1}+\cdots +s_{n+1}^{n+1}}$. We consider
four possibilities and argue each separately. Assume first that
$e_n$, $e_{n+1}$ and $e_{n+2}$ share a common endpoint and that
$e_{n+1}<e_{n+2}$. Then $l_{n+1}=l_ne^{s(e_{n+1})}=e^{s_1^n+\cdots
+s_n^n+s_{n+1}^{n+1}}$. Since there is no additional change of fans
from $e_{n+1}$ to $e_{n+2}$, we have $s_i^n=s_i^{n+1}$ for
$i=1,2,\ldots ,n$. This proves the formula in this case. The second
case is when $e_n$, $e_{n+1}$ and $e_{n+2}$ share a common endpoint
and $e_{n+2}<e_{n+1}$. Then we have
$l_{n+1}=l_ne^{-s(e_{n+1})}=l_ne^{s_{n+1}^{n+1}}$ by the definition
of $s_{n+1}^{n+1}$. The desired formula follows as in the previous
case. In the third case we assume that $e_n$, $e_{n+1}$ and
$e_{n+2}$ do not share a common endpoint and that $e_{n+1}<e_{n+2}$.
Then $l_{n+1}=l_n^{-1}e^{s(e_{n+1})}=e^{-s_1^n-\cdots
-s_n^n+s_{n+1}^{n+1}}$. Since we have one additional change of fan
from $e_{n+1}$ to $e_{n+2}$, we get that $s_i^{n+1}=-s_i^n$ for
$i=1,2,\ldots ,n$. This proves the formula in the third case.
Finally, we assume that $e_n$, $e_{n+1}$ and $e_{n+2}$ do not share
a common endpoint and that $e_{n+2}<e_{n+1}$. Then
$l_{n+1}=l_n^{-1}e^{-s(e_{n+1})}=e^{-s_1^n-\cdots
-s_n^n+s_{n+1}^{n+1}}$. As in the previous case this gives the
desired formula. Therefore the series
$\sum_{n=1}^{\infty}e^{s_1^n+\cdots +s_n^n}$ is the length of
$l(P_1)$ and the proof of Theorem C is completed.
 $\Box$

\section{Quasisymmetric maps and shears}

In this section we characterize shear maps which give rise to
quasisymmetric maps of $\hat{\R}$. This is the main result of the
paper and, to our best knowledge, it gives the only known
parametrization of the universal Teichm\"uller space $T(\H )$.

\vskip .2 cm

\noindent {\it Proof of the first part of Theorem A.} We prove that
the first condition in the theorem is necessary for $s:\F\to\R$ to
be a shear map of a quasisymmetric map $h:\hat{\R}\to\hat{\R}$.

\vskip .2 cm

Consider a fan of $\F$ with tip $p\in\hat{\Q}$.
Let $A\in PSL_2(\Z )$ be such that $A(p)=\infty$. Let $B\in
PSL_2(\R )$ be such that $B(h(p))=\infty$. Then $B\circ h\circ
A^{-1}$ fixes $\infty$ and the corresponding shear map is $s\circ
A^{-1}$. Moreover, if $h$ is $M_1$-quasisymmetric then  $B\circ
h\circ A^{-1}$ is $M$-quasisymmetric, where $M$ is a function of
$M_1$ and is independent of $A$ and $B$. Therefore, we can study
properties of a shear map on a single fan of $\F$ with tip $p$ by studying
properties on the fan of $\F$ with tip $\infty$.

\vskip .2 cm

Consider an $M$-quasisymmetric map $h$ of $\hat{\R}$ which fixes
$\infty$ and let $s:\F\to\R$ be the induced shear map. Then $h$
satisfies
\begin{equation}
\label{quasisymmetry} \frac{1}{M}\leq
\frac{h(n+k)-h(n)}{h(n)-h(n-k)}\leq M
\end{equation}
for all $n\in\Z$ and all $k\in\N$. This is the $M$-quasisymmetric
condition taken at special symmetric triples in $\Z\subset\R$. We
can further normalize $h$ by post-composing with an affine map such
that it fixes $n$, $n+1$ and $\infty$. The values at $\Z$ of such a
normalized $h$ are uniquely determined by shears on the fan of $\F$
with tip $\infty$ by the definition of the characteristic map.

\vskip .2 cm

Let $e_n$ be the geodesic with endpoints $n$ and $\infty$, and let $s_n=s(e_n)$
for the convenience of notation. The condition (\ref{quasisymmetry})
is equivalent to
\begin{equation}
\label{shear_fan_qs} \frac{1}{M}\leq \frac{1+e^{s_{n+1}}+\cdots
+e^{s_{n+1}+s_{n+2}+\cdots
+s_{n+k-1}}}{e^{-s_{n}}+e^{-s_{n}-s_{n-1}}+\cdots
+e^{-s_{n}-s_{n-1}-\cdots -s_{n-k+1}}}\leq M.
\end{equation}

The condition (\ref{shear_fan_qs}) is equivalent to the first
condition in Theorem A and this establishes the necessity of the first
condition in Theorem A.

\vskip .3 cm

We assume that a shear map $s:\F\to\R$ satisfies property
(\ref{shear_fan_qs}) at each fan of $\F$ and show that characteristic map
$h_s:\hat{\mathbf{Q}}\to\hat{\R}$  extends to a quasisymmetric map of $\hat{\R}$.

\vskip .2 cm

We first show that $h_s:\hat{\Q}\to\hat{\R}$ extends to a
homeomorphism of $\hat{\R}$. Since $h_s$ is a strictly monotone
map of $\hat{\Q}$ into $\hat{\R}$, it is enough to show that
$h_s(\hat{\Q})$ is dense in $\hat{\R}$. Assume on the contrary that
$\hat{\R}\setminus h_s(\hat{\Q})$ contains an interval $I$. Assume
that $I$ is a maximal such interval and let $l$ be the geodesic in
$\H$ with endpoints equal to the endpoints of $I$. There are two
possibilities to consider. Either $h_s(\hat{\Q})$ contains exactly
one endpoint of $I$ or both endpoints of $I$ do not lie in
$h_s(\hat{\Q})$.

\vskip .2 cm

In the former case, the interval $I$ has an endpoint $h_s(p)$ for
some $p\in\hat{\Q}$. This implies that the image of the fan at $p$
under $h_s$ accumulates to the geodesic $l\in\H$. Let $C$ be a
horocycle based at $p$. Fix a single geodesic in the fan at
$h_s(p)$. Then the sum of lengths of consecutive arcs of $C$ cut out
by the geodesics in the fan at $h_s(p)$ which accumulate to $l$
starting from the fixed geodesic in the fan is finite. This implies
that there exists a sequence of $2n$ consecutive arcs on $C$ such
that the ratio of the length of left $n$ consecutive arcs to the length of the right $n$
consecutive arcs is converging to $\infty$.
Consequently, the condition (\ref{shear_fan_qs}) fails at the fan with tip
$p$ which is a contradiction.

\vskip .2 cm

In the later case, there is a sequence $e_n\in\F$ such that
$h_s(e_n)$ converges to $l$ and that no $h_s(e_n)$ shares an
endpoint with $l$. Moreover, we can assume that each $e_n$ shares
one endpoint with $e_{n+1}$ namely $\{ e_n\}$ is a chain. We exhibit
a sequence of pairs of adjacent triangles in $h(\F )$ with shears
converging to $0$ or to $\infty$ which again contradicts condition
(\ref{shear_fan_qs}). Let $e_{n_0}$ be such that $h_s(e_{n_0})$ is
close to $l$. Then $e_{n_0+1}$ shares an endpoint with $e_{n_0}$.
Let $e_{n_0+k}$ be the edge in the sequence $\{ e_n\}$ with largest
index which shares an endpoint with $e_{n_0}$. Then $e_{n_0+k+1}$
does not share an endpoint with $e_{n_0+k-1}$ (see Figure
\ref{fig3}). We consider the two adjacent triangle in $\F$ with
common boundary edge $e_{n_0+k}$. The image of the two triangles
under $h_s$ has sides $h_s(e_{n_0+k-1})$, $h_s(e_{n_0+k})$ and
$h_s(e_{n_0+k+1})$ close to the geodesic $l$. This implies that the
other two sides are small in the Euclidean sense. Thus the shear is
very large or very small which is a contradiction with condition
(\ref{shear_fan_qs}). We proved that $h_s$ extends to a
homeomorphism of $\hat{\R}$.

\begin{figure}
\centering
\includegraphics[scale=0.3]{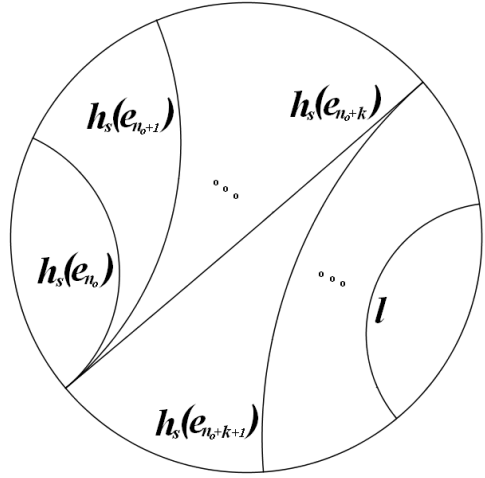}
\caption{The accumulation to $l$.} \label{fig3}
\end{figure}

\vskip .2 cm

It remains to show that $h_s:\hat{\R}\to\hat{\R}$ is a
quasisymmetric map. Let $F_s=ex(h_s)$ be the barycentric extension
of $h_s$ (see Douady-Earle \cite{DE} for the definition). Then
$F_s:\H\to\H$ is a real analytic diffeomorphism of $\H$. The map
$h_s$ is quasisymmetric if and only if $F_s$ is quasiconformal.
Let $\mu_{F_s}=\frac{\bar{\partial}F_s}{\partial F_s}$ be the
Beltrami coefficient of $F_s$.

\vskip .2 cm

Assume on the contrary that $F_s$ is not quasiconformal.  Then
there exists a sequence $z_n\in\H$ such that $|\mu_{F_s}(z_n)|\to
1$ as $n\to\infty$. Since $F_s$ is a real analytic diffeomorphism (see \cite{DE}),
it follows that $z_n$ leaves every compact subset of $\H$. There
are two possibilities for $z_n$. Either there exist a horoball $D$
with center at $\infty$ and a subsequence $z_{n_k}$ of $z_n$ such
that $z_{n_k}$ lies outside the $PSL_2(\Z )$ orbit of $D$, or
sequence $z_n$ enters the $PSL_2(\Z )$ orbit of every horoball
with center at $\infty$.

\vskip .2 cm

Suppose that we are in the former case. For simplicity, denote the
subsequence $z_{n_k}$ by $z_n$ again. Let $\Delta_n$ be triangle in
the complement of $\F$ which contains $z_n$. Let $A_n\in PSL_2(\Z )$
be such that $A_n(\Delta_n)=\Delta_0$. Let $B_n\in PSL_2(\R )$ be
such that $B_n\circ h_s\circ A_n^{-1}$ fixes $0$, $1$ and $\infty$.
By the conformal naturality of the barycentric extension, we have
that $ex(B_n\circ h_s\circ A_n^{-1})=B_n\circ F_s\circ
A_n^{-1}=F_n$. Let $z_n'=A_n(z_n)\in\Delta_0$. Then $z_n'$ belongs
to a compact subset of $\H$ and
$|\mu_{F_n}(z_n')|=|\mu_{F_s}(z_n)|$. Condition (\ref{shear_fan_qs})
implies that individual shears are bounded by $1/M$ from below and
by $M$ from above. This implies that the sequence of shear maps
$s\circ A_n^{-1}$ corresponding to homeomorphisms $B_n\circ h_s\circ
A_n^{-1}$ has a convergent subsequence in the sense that for each
edge $e\in\F$ the sequence of real numbers $s\circ A_{n_k}^{-1}(e)$
converges as $k\to\infty$. The limiting map $s_{\infty}:\F\to\R$
satisfies property (\ref{shear_fan_qs}) in each fan with the
constant $M$ because each $s\circ A_n^{-1}$ does. By the
normalization of $B_n\circ h_s\circ A_n^{-1}$, we get that
$B_{n_k}\circ h_s\circ A_{n_k}^{-1}$ pointwise converges to a
homeomorphism $h_{s_{\infty}}$ of $\hat{\R}$ with shear map
$s_{\infty}$. By the continuity of the barycentric extension, we get
that $|\mu_{F_{n_k}}|$ converges to $|\mu_{ex(h_{s_{\infty}})}|$
uniformly on compact subsets of $\H$. This implies that for a
compact subset $K$ of $\H$ there exists $a<1$ such that
$|\mu_{F_{n_k}}|\leq a$ on $K$. On the other hand, we have that
$|\mu_{F_{n_k}}(z_{n_k})|\to 1$ as $k\to\infty$ which gives a
contradiction.

\vskip .2 cm

Suppose that we are in the later case. Namely,
$|\mu_{F_s}(z_n)|\to 1$ as $n\to\infty$ with $z_n$ entering the
$PSL_2(\Z )$ orbit of every horoball based at $\infty$. Let
$\Delta_n$ be a complementary triangle of $\F$ which contains $z_n$. Let
$A_n\in PSL_2(\Z )$ be such that $A_n(\Delta_n)=\Delta_0$ and that
$A_n(z_n)=z_n'\to\infty$ as $n\to\infty$. Let $B_n\in PSL_2(\R )$
be such that $B_n\circ h_s\circ A_n^{-1}=h_n$ fixes $0$, $1$ and
$\infty$. Then $h_n$ satisfies property (\ref{shear_fan_qs}) with
the same constant $M$ as does $h$. By the conformal naturality of
the barycentric extension, we have that
$|\mu_{F_s}(z_n)|=|\mu_{ex(h_n)}(z_n')|\to 1$ as $n\to\infty$. Let
$\lambda_n=Im(z_n')$ and let $\lambda_n'$ be such that
$\hat{h}_n(x)=\frac{1}{\lambda_n'}h_n(\lambda_nx)$ fixes $1$. It
is clear that $\hat{h}_n$ fixes $0$ and $\infty$ as well. Let
$w_n=\frac{1}{\lambda_n}z_n'$. Then $w_n\to i$ and
$|\mu_{ex(\hat{h}_n)}(w_n)|=|\mu_{ex(h_n)}(z_n')|=|\mu_{F_s}(z_n)|\to
1$ as $n\to\infty$. We need the following lemma in order to finish
the proof.

\begin{lemma}
\label{conv_at-1} Under the above normalization, there exists a
constant $c_0>1$ such that $\frac{1}{c_0}\leq -\hat{h}_n(-1)\leq
c_0$.
\end{lemma}

\begin{proof}
Let $k_n\in\N$ be such that $k_n\leq\lambda_n\leq k_n+1$. Then
$h_n(k_n)\leq h_n(\lambda_n)=\lambda_n'\leq h_n(k_n+1)$. By property
(\ref{shear_fan_qs}), we have that $h_n(k_n+1)-h_n(k_n)\leq
Mh_n(k_n)$. This implies that
\begin{equation}
\label{ineq1}
h_n(k_n+1)\leq (M+1)h_n(k_n)\leq (M+1)\lambda_n'.
\end{equation}

\vskip .2 cm

By applying property (\ref{shear_fan_qs}) to $h_n$ at points
$-(k_n+1)$, $0$ and $k_n+1$, we get that $\frac{1}{M}h_n(k_n+1)\leq
-h_n(-k_n-1)\leq Mh_n(k_n+1)$. Similarly, we get that
$\frac{1}{M}h_n(k_n)\leq -h_n(-k_n)\leq Mh_n(k_n)$. These two
inequalities imply that
$$-Mh_n(k_n+1)\leq h_n(-k_n-1)\leq
h_n(-\lambda_n)\leq h_n(-k_n)\leq -\frac{1}{M}h_n(k_n).$$

\vskip .2 cm

From (\ref{ineq1}), we get
$$
h_n(k_n)\geq\frac{1}{M+1}h_n(k_n+1)\geq\frac{1}{M+1}\lambda_n'.
$$

The above two inequalities and (\ref{ineq1}) give that
$$
-M(M+1)\lambda_n'\leq h_n(-\lambda_n)\leq
-\frac{1}{M(M+1)}\lambda_n'
$$
which implies
$$
-M(M+1)\leq\frac{1}{\lambda_n'}h_n(-\lambda_n)=\hat{h}_n(-1)\leq
-\frac{1}{M(M+1)}.
$$
Take $c_0=M(M+1)$ and the above becomes $\frac{1}{c_0}\leq
-\hat{h}_n(-1)\leq c_0$.
\end{proof}

\vskip .2 cm

We finish the proof using the above lemma. Note that $\hat{h}_n$
fixes $0$, $1$ and $\infty$, and $\hat{h}_n(-1)$ is bounded away
from $0$ and $\infty$ by the above lemma. Then Lemma \ref{baryc}
implies that $|\mu_{\hat{h}_n}|\leq c<1$ in a neighborhood of
$i\in\H$ and for all $n\in\N$ (see also \cite[Lemma 3.6]{Mar},
\cite{AEM}, \cite{DE}). On the other hand, the assumption on $w_n$
and conformal naturality of barycentric extension implies that
$|\mu_{\hat{h}_n}(w_n)|\to 1$ as $n\to\infty$ which is a
contradiction. This finishes the proof of the first statement in
Theorem A. $\Box$

\vskip .2 cm

\noindent {\it Proof of the second part of Theorem A.} Consider a
fan of geodesics of $\F$ with tip $p\in\hat{\mathbf{Q}}$ and assume
that $e_n\in\F$, $n\in\Z$, is a fixed correspondence with $\Z$
induced by natural ordering as before. Let $a_n\in\hat{\Q}$ be the
endpoint of $e_n$ that is different from $p$. Then $(a_k,a_m,a_n,p)$
are in the cyclic order of $\hat{\R}$ if $k<m<n$. The triple
$a_k,a_m,a_n$ is said to be {\it fan-symmetric} if $m-k=n-m$. The
point $a_m$ is said to be the {\it midpoint} of the triple.

\vskip .2 cm

Let $a_k,a_m,a_n\in\hat{\Q}$ be a fan-symmetric triple for the fan
with tip $p\in\hat{\Q}$, where $a_m$ is the mid-point of the
triple. This implies that
$cr(p,a_k,a_m,a_n)=\frac{(a_m-p)(a_n-a_k)}{(a_n-p)(a_m-a_k)}=2$.
The {\it generation of a triple} $e_k,e_m,e_n$ of geodesics is the
minimum of the Farey generations of $e_k$ and $e_n$ . Let
$h:\hat{\R}\to\hat{\R}$ be a symmetric map which fixes $0$, $1$
and $\infty$. If the generation of a triple $e_k,e_m,e_n$ is
large, it follows that the points $a_k$, $p$ and $a_n$ are close
in the angle metric of $\hat{\R}$ with respect to $i\in\H$. The
barycentric extension $ex(h)=F$ of $h$ has Beltrami coefficient
close to zero in a definite Euclidean neighborhood in $\H$ of the
triple $(a_k,p,a_n)$ (see \cite{EMS}). A length-area argument
implies that $cr(h(p),h(a_k),h(a_m),h(a_n))$ is close to $2$ (see,
for example, \cite{GS}). After post-composing $h$ by $A\in
PSL_2(\R )$ such that $A\circ h(a_m)=\infty$, this is equivalent
to the fact that the ratio $\frac{|A\circ h(p)-A\circ
h(a_k)|}{|A\circ h(a_n)-A\circ h(p)}$ is close to $1$. Let
$s:\F\to\R$ be the shear map of $h$ and let $s_i=s(e_i)$. Then for
a given $\epsilon
>0$, there exists $k=k(\epsilon )\in\N$ such that on any fan-symmetric triple of
generation at least $k$ the shear map $s:\F\to\R$ satisfies
\begin{equation}
\label{shear_fan_sym} \frac{1}{1+\epsilon}\leq
\frac{1+e^{s_1}+\cdots +e^{s_1+s_2+\cdots
+s_n}}{e^{-s_0}+e^{-s_0-s_{-1}}+\cdots +e^{-s_0-s_{-1}-\cdots
-s_{-n}}}\leq 1+\epsilon .
\end{equation}
Thus we established the necessity of the second condition in Theorem
A.

\vskip .2 cm

We show that the second condition in Theorem A is also sufficient
for a map to be symmetric. For any $k\in\N$, there are only finitely
many geodesics in $\F$ whose generation is at most $k$. Together
with (\ref{shear_fan_sym}), this implies that the shear map
$s:\F\to\R$ is bounded and that $s(e)$ converges to $0$ as the
generation of $e$ converges to $\infty$, where the speed of
convergence depends only on the generation of $e\in\F$. The cocycle
map $h_s$ of the shear map $s$ with property (\ref{shear_fan_sym})
extends to a homeomorphism of $\hat{\R}$. The proof follows the same
lines as in the proof of the first part of Theorem A and we omit it
here.

\vskip .2 cm

It remains to show that $h_s$ is a symmetric map. We consider the
barycentric extension $ex(h_s)=F_s$ of $h_s$. It is enough to show
that $F_s$ is an asymptotically conformal map of $\H$ (see \cite{EMS}).

\vskip .2 cm

Assume on the contrary that there exists a sequence $z_n\in\H$
which leaves every compact subset of $\H$ such that
$|\mu_{F_s}(z_n)|\geq c>0$. Let $\Delta_n$ be the ideal triangle
in $\F$ which contains $z_n$. Let $A_n\in PSL_2(\Z )$ be such that
$A_n(\Delta_n)=\Delta_0$, where $\Delta_0$ is the triangle in $\F$
with vertices $0$, $1$ and $\infty$. Let $B_n\in PSL_2(\R )$ be
such that $h_n=B_n\circ h_s\circ A_n^{-1}$ fixes $0$, $1$ and
$\infty$. Let $z_n'=A_n(z_n)\in\Delta_0$ and let $F_n=ex(h_n)$ be
the barycentric extension of $h_n$.

\vskip .2 cm

Assume first that a subsequence of $z_n'$ stays in a compact part
of $\Delta_0$, and for simplicity we denote the subsequence by
$z_n'$ again. This implies that the sequence of $\Delta_n$'s
contains infinitely many pairwise different triangles because
$z_n$ leave any compact subset of $\H$. In particular, the minimum
of the generations of the edges of $\Delta_n$ converges to
infinity as $n\to\infty$. Consequently, shear maps $s\circ
A_n^{-1}$ converge to the zero map which implies that $h_n$
converges pointwise on $\hat{\R}$ to the identity. On the other
hand, $|\mu_{F_n}(z_n')|\geq c>0$ by conformal naturality of the
barycentric extension. This is a contradiction with the continuity
properties of the barycentric extension (see \cite{DE} or Section
2).

\vskip .3 cm

In the other case, we assume that $z_n'\to\infty$ inside $\Delta_0$
as $n\to\infty$. Let $\lambda_{n,1}=[Im(z_n')]$ be the greatest
integer less than or equal to $Im(z_n')$. Clearly
$\lambda_{n,1}\to\infty$ as $n\to\infty$. Let
$\lambda_{n,1}'=h_n(\lambda_{n,1})$. Define
$\frac{1}{\lambda_{n,1}'}h_n(\lambda_{n,1}x)=\tilde{h}_{n,1}(x)$
and note that $\tilde{h}_{n,1}(x)$ fixes $0$, $1$ and $\infty$.

\vskip .2 cm

Let $x,y,z\in\Z$ be symmetric points such that
$\tilde{h}_{n,1}(x)\to x$ and $\tilde{h}_{n,1}(y)\to y$ as
$n\to\infty$. If either $x>y>z$ and $x\neq 0$, or $x<y<z$ and
$x\neq 0$, or $x<z<y$ and $x,y\neq 0$, then we claim that
$\tilde{h}_{n,1}(z)\to z$ as $n\to\infty$. We prove this when
$x>y>z$ and $x\neq 0$. Other cases are similar and they are left
to the reader. For $z=0$ we have immediately that
$\tilde{h}_{n,1}(z)=z$. We assume now that $z\neq 0$. Since $x\neq
0$ we have that $\lambda_{n,1}x\to\pm\infty$ and
$\lambda_{n,1}z\to\pm\infty$ as $n\to\infty$. This implies that
the generation of the symmetric triples $e_{\lambda_{n,1}x}$,
$e_{\lambda_{n,1}y}$ and $e_{\lambda_{n,1}z}$ goes to infinity as
$n\to\infty$. Then the condition (\ref{shear_fan_sym}) implies
that
$\frac{\tilde{h}_{n,1}(x)-\tilde{h}_{n,1}(y)}{\tilde{h}_{n,1}(y)-\tilde{h}_{n,1}(z)}\to
1$ as $n\to\infty$ because
$\lambda_{n,1}x,\lambda_{n,1}y,\lambda_{n,1}z\in\Z$ and
$h_n(\lambda_{n,1}x),h_n(\lambda_{n,1}y),h_n(\lambda_{n,1}z)$
depend only on the shears at the fan with tip $\infty$. This gives
$\tilde{h}_{n,1}(z)\to z$ as $n\to\infty$.

\vskip .2 cm

Recall that $\tilde{h}_{n,1}$ fixes $0$, $1$ and $\infty$. We use
the statement in the above paragraph to show that
$\lim_{n\to\infty}\tilde{h}_{n,1}(k)=k$ for all $k\in\Z$. Using
the triple $-1$, $0$ and $1$, we get that
$\lim_{n\to\infty}\tilde{h}_{n,1}(-1)=-1$. Then using the triple
$-1$, $1$ and $3$ we get that
$\lim_{n\to\infty}\tilde{h}_{n,1}(3)=3$. The triple $1$, $2$ and
$3$ gives that $\lim_{n\to\infty}\tilde{h}_{n,1}(2)=2$. Then using
the triple $2$, $3$ and $4$ gives the convergence for $4$, and
continuing like this we get the convergence
$\lim_{n\to\infty}\tilde{h}_{n,1}(k)=k$ for all $k\in\Z^{+}$.
Similarly, we get $\lim_{n\to\infty}\tilde{h}_{n,1}(k)=k$ for all
$k\in\Z^{-}$.

\vskip .2 cm

Let $\lambda_{n,r}$ be the greatest integer multiple of $2^{r-1}$
which is less than or equal to $Im(z_n')$ for $r\in\N$. Clearly
$\lambda_{n,r}\to\infty$ as $n\to\infty$. Let
$\lambda_{n,r}'=h_n(\lambda_{n,r})$. Define
$\frac{1}{\lambda_{n,r}'}h_n(\lambda_{n,r}x)=\tilde{h}_{n,r}(x)$ and
note that $\tilde{h}_{n,r}(x)$ fixes $0$, $1$ and $\infty$. We claim
that $\lim_{n\to\infty}\tilde{h}_{n,r}(k/2^{i})=k/2^{i}$ for all
$k\in\Z$ and $i=0,\ldots ,r-1$. For a fixed $r$, the proof is by
finite induction on $i$. The case $i=0$ is proved in the above
paragraph. Assume that the statement is true for $i$ and we need to
prove that it is true for $i+1$. The inductive hypothesis says that
$\lim_{n\to\infty}\tilde{h}_{n,r}(k/2^{i})=k/2^{i}$ for $k\in\Z$
because $\lambda_{n,r}\frac{k}{2^{i}}\in\Z$ for each $n\in\N$. Since
each $k/2^{i+1}$, for $k\in\Z$ odd, is in the middle of
$(k-1)/2^{i+1}$ and $(k+1)/2^{i+1}$ on which the convergence holds
and since $\lambda_{n,r}\frac{k}{2^{i+1}}\in\Z$, it follows similar
to the above that
$\lim_{n\to\infty}\tilde{h}_{n,r}(k/2^{i+1})=k/2^{i+1}$. This
finishes the induction.

\vskip .2 cm

We use the Cantor diagonalization process to obtain a contradiction.
The set $D=\{ k/2^{r-1}:r\in\N , k\in\Z\}$ is a dense subset of
$\hat{\R}$. We put $D$ into a sequence $\{ b_m\}$ such that if
$b_m=k/2^{r-1}$ for minimal $r\in\N$ then $m\geq r$. Fix $m\in\N$.
Then there exists $n_m$ such that $|\tilde{h}_{n_m,m}(b_i)-b_i|<1/m$
for $i=1,2,\ldots ,m$ and $|z_m'/\lambda_{n_m,m}-i|<1/m$. This
implies that $\tilde{h}_{n_m,m}$ converges pointwise to the identity
on $\hat{\R}$. On the other hand, the Beltrami coefficient of
$ex(\tilde{h}_{n_m,m})$ at $\frac{z_m'}{\lambda_{n_m,m}}$ is bounded
away from $0$ by conformal naturality of the barycentric extension.
This is a contradiction. Therefore $h_s$ is symmetric which finishes
the proof of Theorem A. $\Box$

\section{The topology on $\X$}

Let $\X$ be the space of all shear maps $s:\F\to\R$ which satisfy
condition (\ref{shear_fan_qs}) on each fan of geodesics in $\F$ with
the same constant. Theorem A implies that the universal
Teichm\"uller space $T(\H )$ is parameterized by the space $\X$. We
turn our attention to the topology on $\X$ which would make the map
$T(\H )\to\X$ a homeomorphism.

\vskip .2 cm

Consider a shear map $s\in\X$ and a fan of geodesics in $\F$ with
tip $p$. Let $e_n$, $n\in\Z$, be the enumeration of the fan. For a
given horocycle $C$ with center $p$, we denote by $s(p;n,k)$ the
quotient of the length of arc of $C$ between $h_s(e_{n+k})$ and
$h_s(e_n)$ to the length of the arc of $C$ between $h_s(e_n)$ and
$h_s(e_{n-k})$, for $n,k\in\Z$. Note that $s(p;n,k)$ is the
expression in the middle of (\ref{shear_fan_qs}) described in a
coordinate independent fashion.

\vskip .2 cm

Let $M(s)\geq 1$ be the supremum of $s(p;n,k)$ over all
$p\in\hat{\Q}$, $n,k\in\Z$. If $M(s)<\infty$, then we say that
$s:\F\to\R$ satisfies $M(s)$-shear condition. For example, the shear
map $s_{id}$ of the basepoint $id\in T(\H )$ is assigning $0$ to
each edge of $\F$ and $M(s_{id})=1$.

\vskip .2 cm

More generally, let $s_1,s_2\in\X$. Define $M(s_1,s_2)$ to be the
supremum of the maximum of $s_1(p;n,k)/s_2(p;n,k)$ and $s_2(p;n,k)/s_1(p;n,k)$ over
all $p\in\hat{\Q}$, $n\in\Z$ and $k\in\N$. Note that
$M(s_1,s_2)=M(s_2,s_1)$ and that $M(s_1,s_{id})=M(s_1)$.

\vskip .2 cm

Let $h_n:\hat{\R}\to\hat{\R}$ be a sequence of quasisymmetric maps
which fix $0$, $1$ and $\infty$, and which converge to the identity
in the Teichm\"uller topology in $T(\H )$. Then we immediately
obtain that $M(s_{h_n})\to 1$ as $n\to\infty$ from the
quasisymmetric condition.

\vskip .3 cm

\noindent {\it Proof of Theorem B.} Recall that $h_n\to id$ in the
Teichm\"uller topology if and only if
$\sup\frac{cr(h_n(a),h_n(b),h_n(c),h_n(d))}{cr(a,b,c,d)}\to 1$ as
$n\to\infty$, where the cross-ratio is
$cr(a,b,c,d)=\frac{(c-a)(d-b)}{(d-a)(c-b)}$ and the supremum is over
all quadruples $(a,b,c,d)\in (\hat{\R})^4$ with the cross-ratio
between $1+1/M$ and $1+M$ for some $M>1$. By the definition, $h_n\to
h$ as $n\to\infty$ in the Teichm\"uller topology if and only if
$h_n\circ h^{-1}\to id$ in the Teichm\"uller topology. A quadruple
of points in $\hat{\Q}$ with cross-ratio $2$ where one point is the
tip of the fan such that the other three points are endpoints of
geodesics in the fan is {\it fan-symmetric} (see proof of Theorem A
for equivalent definition). The cross-ratio of the image under $h$
of a fan-symmetric quadruple is bounded away from $1$ and $\infty$
because $h$ is quasisymmetric. The above characterization of the
Teichm\"uller topology when applied to $h_n\circ h^{-1}\to id$ at
the images under $h$ of all fan-symmetric quadruples gives that
$M(s_{h_n},s_h)\to 1$ as $n\to\infty$. This proves the necessity of
the condition.

\vskip .2 cm

Given $h,h_n\in T(\H )$ such that $M(s_{h_n},s_h)\to 1$ as
$n\to\infty$, we need to show that $h_n\to h$ as $n\to\infty$.
Assume on the contrary that $h_n$ does not converge to $h$ in the
Teichm\"uller topology. Let $F=ex(h)$ and $F_n=ex(h_n)$ be the
barycentric extensions of $h$ and $h_n$, respectively. The
assumption implies that there exists $c>0$ and a sequence
$z_n\in\H$ such that $|\mu_F(z_n)-\mu_{F_n}(z_n)|\geq c$. There
are two possibilities for the sequence $z_n$. Either there exists
a horoball $C$ with center $\infty$ and a subsequence $z_{n_k}$
such that $z_{n_k}$ is disjoint from the $PSL_2(\Z )$ orbit of
$C$, or for any horoball $C$ with center $\infty$ only finitely
many $z_n$'s lie outside the $PSL_2(\Z )$ orbit of $C$.

\vskip .2 cm

Assume we are in the former case. For the convenience of notation,
replace $z_{n_k}$ with $z_n$. Let $A_n\in PSL_2(\Z )$ be such that
$A_n(\Delta_n)=\Delta_0$, where $\Delta_n$ is a complementary
triangle of $\F$ which contains $z_n$. Then $A_n(z_n)$ lies in a
compact subset of $\H$. Let $B_n,B_n^{*}\in PSL_2(\R )$ be such
that $B_n\circ h\circ A_n^{-1}$ and $B_n^{*}\circ h_n\circ
A_n^{-1}$ fix $0$, $1$ and $\infty$. Since $M(s_{h_n},s_h)\to 1$
as $n\to\infty$, we get that $B_n\circ h\circ A_n^{-1}$ and
$B_n^{*}\circ h_n\circ A_n^{-1}$ pointwise converge to the same
quasisymmetric map. Therefore the Beltrami coefficients of their
corresponding barycentric extensions converge uniformly on compact
subsets of $\H$ to the same Beltrami coefficient (see \cite{DE}).
This contradicts $|\mu_F(z_n)-\mu_{F_n}(z_n)|\geq c$.

\vskip .2 cm

Assume we are in the later case. Let $A_n\in PSL_2(\Z )$ and
$B_n,B_n^{*}\in PSL_2(\R )$ be as above. Let $h_n'=B_n\circ h\circ
A_n^{-1}$ and $h_n^{*}=B_n^{*}\circ h_n\circ A_n^{-1}$. In addition,
we may assume that $A_n(z_n)\to\infty$ as $n\to\infty$. To find a
contradiction in this case, we use the idea from the proof of the
last part of Theorem A. Denote by $\lambda_{n,r}$ the largest
integer multiple of $2^{r-1}$ which is less than or equal to
$Im(z_n)$. Let $\lambda_{n,r}'=h_n'(\lambda_{n,r})$ and
$\lambda_{n,r}^{*}=h_n^{*}(\lambda_{n,r})$. Then
$\tilde{h}_{n,r}'(x)=\frac{1}{\lambda_{n,r}'}h_n'(\lambda_{n,r} x)$
and
$\tilde{h}_{n,r}^{*}(x)=\frac{1}{\lambda_{n,r}^{*}}h_n^{*}(\lambda_{n,r}
x)$ fix $0$, $1$ and $\infty$. For each $r\in\N$, sequences
$\tilde{h}_{n,r}'(x)$ and $\tilde{h}_{n,r}^{*}(x)$ have convergent
subsequences (in the pointwise sense) whose limits $h^1_r$ and
$h^2_r$ agree on the set $\{ k/2^{r-1}:k\in\Z\}$ because
$M(s_{h_n},s_h)\to 1$ as $n\to\infty$. The values of maps $h^1_r$
and $h^2_r$ on $\{ k/2^{r-1}:k\in\Z\}$ depend only on the shears of
$h_n'$ and $h_n^{*}$ on the fan with tip $\infty$. Using the Cantor
diagonalization process, we find sequences $\tilde{h}_{n_m,m}'(x)$
and $\tilde{h}_{n_m,m}^{*}(x)$ whose pointwise limits $h^1$ and
$h^2$ satisfy $h^1=h^2$ and $\frac{z_m'}{\lambda_{n_m,m}}\to i\in\H$
as $m\to\infty$. This again gives a contradiction with
$|\mu_F(z_m)-\mu_{F_m}(z_m)|\geq c$ by conformal naturality of the
barycentric extension. $\Box$

\section{Decorated tesselations and lambda lengths}

A {\it tesselation} $\tau$ of $\H$ is a locally finite countable
geodesic lamination of $\H$ such that the components in
$\H\setminus\tau$ are ideal hyperbolic triangles. A {\it decorated
tesselation} $\tilde{\tau}$ is a tesselation $\tau$ of $\H$
together with an assignment of a horocycle to each vertex of
$\tau$ whose center is that vertex (see \cite{Pe}).

\vskip .2 cm

Let $\tau$ be a tesselation with a distinguished oriented edge
$e=(x_i,x_t)$, where $x_i$ is the initial point and $x_t$ is the
terminal point of $e$. Recall that $\F$ is the Farey tesselation
and let $(-1,1)$ be a distinguished oriented edge of $\F$. Denote
by $\tau^{0}$ the set of vertices of $\tau$. Recall that
$\hat{\Q}\subset\hat{\R}$ is the set of vertices of $\F$. There
exists a unique map $h_{\tau}:\hat{\Q}\to\tau^{0}$ such that
$h_{\tau}(x_i)=-1$, $h_{\tau}(x_t)=1$ and that if
$x,y,x\in\hat{\Q}$ are vertices of a complementary triangle of
$\F$ then $h_{\tau}(x),h_{\tau}(y),h_{\tau}(z)\in\tau^0$ are the
vertices of a complementary triangle of $\tau$ (see \cite{Pe}). We
call $h_{\tau}$ the {\it characteristic map} of $\tau$. It is
clear that $h_{\tau}:\hat{\Q}\to \hat{\R}$ extends by continuity
to a homeomorphism of $\hat{\R}$ because $\hat{\Q},\tau^0$ are
dense in $\hat{\R}$ and $h_{\tau}$ is monotone on $\hat{\Q}$.

\vskip .2 cm

Given a decorated tesselation $\tilde{\tau}$ together with a
distinguished oriented edge $e\in\tau$, Penner \cite{Pe1} assigns to
each edge $f\in\F$ a positive number as follows. Let $C_1$ and $C_2$
be horocycles of the decoration $\tilde{\tau}$ based at the
endpoints of $h_{\tau}(f)\in\tau$. Let $\delta (f)$ be a signed
hyperbolic distance between $M_1=h_{\tau}(f)\cap C_1$ and
$M_2=h_{\tau}(f)\cap C_2$, where the sign is positive if the arc of
$h_{\tau}(f)$ between $M_1$ and $M_2$ is outside $C_1$ and $C_2$,
otherwise the sign is negative (see \cite{Pe1}). The {\it lambda
length } of $f\in\F$ is given by
$$
\lambda (f)=e^{-2\delta (f)}.
$$
This introduces the {\it lambda length} function $\lambda
:\F\to\R^{+}$ for any decorated tesselation $\tilde{\tau}$ (see
Penner \cite{Pe}). Let $e,e'\in\tau$ be adjacent edges. Then
$h_{\tau}:\F\to\tau$ maps adjacent edges $f,f'\in\F$ onto $e,e'$,
respectively. We define {\it horocyclic length} $\alpha (f,f')$ to
be the length of the arc of the horocycle from $\tilde{\tau}$ with
center the common endpoint of $e$ and $e'$ that lies inside the
hyperbolic wedge with boundary sides $e,e'$.

\vskip .2 cm

Conversely, given a map $\lambda :\F\to\R$ there exists a monotone
map $h_{\lambda}:\hat{\Q}\to \hat{\R}$, called the {\it
characteristic map} of $\lambda$, and a decoration (i.e. choice of
horocycles) on $h_{\lambda}(\hat{\Q})$ such that the lambda length
of $h_{\lambda}(f)$ with respect to the decoration is equal to
$\lambda (f)$. The characteristic map $h_{\lambda}:\hat{\Q}\to
\hat{\R}$ does not always extend to a homeomorphism similar to the
case of shears. It is a fundamental question in this theory to give
necessary and sufficient condition on the map $\lambda :\F\to\R$
such that $h_{\lambda}$ extends by continuity to a homeomorphism or
perhaps to a quasisymmetric map. Penner and Sullivan \cite[Theorem
6.4]{Pe} gave a sufficient condition on the lambda lengths to induce
a quasisymmetric map as follows. A lambda length function $\lambda
:\F\to\R$ is said to be {\it pinched} if there exists $K>1$ such
that
$$
\frac{1}{K}\leq\lambda (f)\leq K,
$$
for all $f\in\F$. Penner and Sullivan showed that if $\lambda
:\F\to\R$ is pinched then the characteristic map $h_{\lambda}$
extends to a quasisymmetric map of $\hat{\R}$ \cite[Theorem
6.4]{Pe}.

\vskip .2 cm

In Theorem E we give a necessary and sufficient condition such that $h_{\lambda}$ is a quasisymmetric (as
well as a symmetric) map of $\hat{\R}$.

\vskip .2 cm

\noindent {\it Proof of Theorem E.} Let $\tau =h_{\lambda}(\F )$ be
the geodesic lamination corresponding to the lambda lengths
$\lambda$ and let $\tilde{\tau}$ be the decorations at the vertices
of $\tau$ corresponding to $\lambda$ (see Penner \cite{Pe} for the
construction). Let $s:\F\to\R$ be shear map corresponding to
$h_{\lambda}$. Let $e_n\in\F$, $n\in\Z$, be a fan of geodesics with
tip $p$. Then we have
$$
s(p;n,k)=\frac{\alpha (e_m,e_{m+1})+\alpha (e_{m+1},e_{m+2})+\cdots +\alpha (e_{m+k},e_{m+k+1})}
{\alpha (e_m,e_{m-1})+\alpha (e_{m-1},e_{m-2})+\cdots +\alpha (e_{m-k},e_{m-k-1})}.
$$
Theorem E immediately follows from Theorem A. $\Box$

\vskip .2 cm

In Theorem D, we find a necessary and sufficient condition such that
$h_{\lambda}$ extends to a homeomorphism of $\hat{\R}$. The criteria
follows from the proof of Theorem C and it is obtained by
calculating the length of $l(P_1)$ in terms of horocyclic and lambda
lengths. Since the horocyclic lengths are expressed in terms of the
lambda lengths, the formula can be written only in terms of the
lambda lengths although we do not do this.

\vskip .2 cm

\noindent {\it Proof of Theorem D.}
Let $\lambda :\F\to\R$ be an assignment of lambda lengths and let
$h_{\lambda}:\hat{\Q}\to\hat{\R}$ be the characteristic map. Denote by
$\tau$ the image tesselation $h_{\lambda}(\F )$ and by
$\tilde{\tau}$ the decoration which realizes the lambda lengths
$\lambda$.

\vskip .2 cm

Let $e_n$, for $n\in\N$, be an arbitrary chain in $\F$. Denote by
$\lambda_n=\lambda (e_n)$ the lambda length of $e_n$. Then
$\lambda_n=e^{-2\delta_n}$, where $\delta_n$ is the signed
hyperbolic distance between the horocycles of $\tilde{\tau}$ with
centers at the endpoints of $e_n$. Thus
$\lambda_n^{-1/2}=e^{\delta_n}$. Let $W_n$ be the wedge with
boundary sides $h_{\lambda}(e_n)$ and $h_{\lambda}(e_{n+1})$ and
let $C_n$ be the horocycle of the decoration $\tilde{\tau}$ with
center at the common endpoint of $h_{\lambda}(e_n)$ and
$h_{\lambda}(e_{n+1})$. Let $\alpha_n$ be the horocyclic length
for the wedge with boundaries $e_n$ and $e_{n+1}$ namely the
length of $C_n\cap W_n$. Let $l_n$ be the length of $l(P_1)\cap
W_n$, where $P_1$ is chosen such that
$l_1=\lambda_1^{-\frac{1}{2}}\alpha_1 =e^{\delta_1}\alpha_1$.

\vskip .2 cm

We need to show that
$l_n=(\lambda_n^{-\frac{1}{2}}\lambda_{n-1}^{\frac{1}{2}}\cdots
\lambda_1^{\frac{(-1)^n}{2}})\alpha_n$. An elementary hyperbolic
considerations shows that $l_n=e^{d_n}\alpha_n$ where $d_n$ is the
signed distance from $l(P_1)\cap W_n$ to the horocycle $C_n$ such
that $d_n>0$ if $l(P_1)\cap W_n$ is outside $C_n$ and that
otherwise $d_n<0$. Therefore it remains to show that
$e^{d_n}=\lambda_n^{-\frac{1}{2}}\lambda_{n-1}^{\frac{1}{2}}\cdots
\lambda_1^{\frac{(-1)^n}{2}}$.

\vskip .2 cm

We finish the argument by induction on $n$. By our choice of
$P_1$, we have immediately that
$e^{d_1}=e^{\delta_1}=\lambda_1^{-\frac{1}{2}}$. Assume that $n>1$
and that
$e^{d_n}=\lambda_n^{-\frac{1}{2}}\lambda_{n-1}^{\frac{1}{2}}\cdots
\lambda_1^{\frac{(-1)^n}{2}}$. We calculate $e^{d_{n+1}}$. Since
$d_n$ is the signed distance from $l(P_1)\cap W_n$ to $C_n$, it
follows that the signed distance of $l(P_1)\cap H_s(e_{n+1})$ to
$C_n$ is $d_n$. Since $\delta_{n+1}$ is the signed distance
between $C_n$ and $C_{n+1}$, it follows that
$d_{n+1}=\delta_{n+1}-d_n$. This gives the desired formula. $\Box$

\end{document}